\documentclass[11pt,amsmath,amsthm,amsfonts,amssymb,amscd]{article}
\usepackage{epsfig}
\usepackage{amsthm}
\usepackage{amsmath}
\usepackage{graphicx}
\usepackage{amssymb}
\usepackage{amscd}
\usepackage[usenames]{color}

\font \Bbbten=msbm10 \font \Bbbsev=msbm7 \font \Bbbfiv=msbm5
\newfam\Bbbfam \textfont\Bbbfam = \Bbbten
\scriptfont\Bbbfam=\Bbbsev \scriptscriptfont\Bbbfam=\Bbbfiv
\newcommand{\N}{\mbox{$I\!\!N$}}
\newcommand{\Z}{\mbox{$Z\!\!\!Z$}}

\newcommand{\R}{\mbox{$I\!\!R$}}

\newcommand{\cO}{\mathcal{O}}

\newcommand{\dist}{{\rm dist}}

\newcommand{\cita}[7]{{\sc #1, }{\it #2, }{\small #3, {\bf #4 } (#5), p.
#6-#7.}}
\newcommand{\cit}[5]{{\sc #1, }{\it #2, }{\small #3, { #4 } #5.}}

\theoremstyle{plain}
\newtheorem{thm}{Theorem}[section]

\newtheorem{maintheorem}{Theorem}

\newtheorem{Df}{Definition}[section]
\newtheorem{Teo}{Theorem}[section]
\newtheorem{Lem}[Teo]{Lemma}

\newtheorem{Prop}[Teo]{Proposition}
\newtheorem{Obs}[Teo]{Remark}

\title{Lyapunov exponents,
expansiveness and entropy \\ for set-valued maps}
\author{
M. J. Pacifico\footnote{partially supported by
 was partially supported by CAPES -- Finance Code 001, CNPq grant 302565/2017-5 and FAPERJ grant CNE 239069},
J. L. Vieitez\footnote{partially supported by Grupo de Investigaci\'on "Sistemas Din\'amicos" CSIC (Universidad de la
Rep\'ublica), SNI-ANII, PEDECIBA, Uruguay}}

\date{\today}

\begin{document}
\maketitle

\begin{abstract}
In this paper, we define Lyapunov exponents for continuous set-valued maps defined on a Peano space, give a notion of expansiveness for a set-valued map $F:X \multimap X$ defined on a topological space $X$ different from that given by Richard Williams, and prove that the topological entropy of an expansive set-valued map defined on a Peano space of positive dimension is strictly positive.
We define the Lyapunov exponent for set-valued maps and prove that the positiveness of its Lyapunov exponent implies positiveness for the topological entropy.
\end{abstract}

\tableofcontents

\section{Introduction}
In this paper, we introduce a notion of expansiveness for set-valued maps different from that given by Richard Williams,  {and} define Lyapunov exponents for set-valued maps. %, and introduce Lyapunov stable points for set-valued maps.
We prove that expansive set-valued maps defined on non trivial Peano spaces have positive entropy and
 if a set-valued map has positive Lyapunov exponents then its entropy is also strictly positive.
Set-valued maps, also called multivalued functions or maps, (see Section \ref{setvalued} for precise definitions), appear naturally in Control Systems Theory. Starting from a state $x\in X$ at a given time $t_0$, at a time $t>t_0$ there will be a subset of the phase space $X$ of possible states that the given system can attain %the
{at} time $t$, see for instance \cite{Ro, Ba}. Also in Mathematical Economics, there are several applications of set-valued maps (see for instance \cite{Au}, where "multivalued maps" are called "correspondences"). Another source of set-valued maps is given by differential equations or differential inclusions where there is no uniqueness of solutions as, for instance, it occurs in the differential equation given by $\dot x =x^{2/3}$, $x\in\R$, for the initial value $x(t_0)=0$, $t_0\in\R$. Moreover, when we have a non invertible onto map $f:X\to X$, the inverse relation $f^{-1}:X\to 2^X$ given by $f^{-1}(y)=\{x\in X\, :\, f(x)=y\}$ gives a particular example of a set-valued map.
Even if $f:M\to M$ is a homeomorphism and we simulate numerically the dynamics given by $f$, a point $x$ usually is known only approximately due, for instance, to experimental errors. Hence,  instead of capturing information at the exacted point $x$, we capture it at a point $\widetilde x$, %WHICH
which adds some error to $x$. Moreover, when we use computer devices, $\widetilde x$ is changed by $fl(\widetilde x)$ where $fl(\widetilde x)$ is the floating point representation of $\widetilde x$ (more precisely the floating point representation of the coordinates of $\widetilde x$). Hence the true orbit $\cdots f^{-1}(x)\mapsto x\mapsto f(x)\mapsto f^2(x)\cdots$ is replaced by some kind of pseudo-orbit.
Thus, since the single value $f(x)$ is only known in a fuzzy way, it seems natural to replace it with a set $F(x)$ leading to a set-valued map.

Recall that in the case of a diffeomorphism $f:M\to M$ defined on a
compact manifold $M$
the Lyapunov exponent at a point $x\in M$ in the direction $v\in T_x M$
is given by $\chi(x,v)= \limsup_{n\to\infty}\frac{1}{n}\log(\|Df^n_x(v)\|)$.
We can interpret $v$ as the ''error'' $v=\exp_x^{-1}(\widetilde{x})$
via the inverse of the exponential map $\exp_x:T_xM\to M$ and $Df_x(v)$ as the linear approximation of the error $\exp_x^{-1}(f(\widetilde{x}))-\exp_x^{-1}(f(x))$. By forward iteration by $Df$, we try to capture the divergence between the trajectories. A positive value of $\chi(x,v)$ implies exponential divergence of trajectories.
Thus, the positiveness of the Lyapunov exponents
indicates the existence of chaos, that is, the long term behavior
of the system is unpredictable from the initial data.

 In \cite{PV}  we have addressed the problem of defining Lyapunov exponents for an expansive homeomorphism $f$ on a compact metric space $(X,\dist)$ using similar techniques as those developed in \cite{BS,Kif}. Under certain conditions about the topology of the space $X$ where $f$ is acting (mainly we require $X$ to be a non trivial Peano space), we obtain that the Lyapunov exponents are different from zero, an indication of chaos.

\medbreak

In this paper,
we first introduce a notion of expansiveness for set-valued maps different from that given by Richard Williams, \cite{Wi, Wi2}, see
Definition \ref{expan}, and prove the following result

\begin{maintheorem}\label{t-expansive-entropy-positive}
Let $M$ be a compact connected manifold and
 $F:M\multimap M$ be an expansive set-valued map such that $F(x)$ is compact and connected for all $x\in M$.
Then $h_{top}(F)>0$.
\end{maintheorem}

Note that given a set-valued map  $F:M\multimap M$ and $x\in M$, the subset $F(x)$ can be a finite or countable subset or  a Cantor set, for instance, etc. So we cannot expect to have that if $\gamma$ is connected then $F(\gamma)$ is connected without any other condition. This leads us to give  direct
 proof that if $M$ is a compact manifold and $F:M\multimap M$ is a continuous expansive
set-valued map satisfying some additional hypothesis then $h_{top}(F)>0$, covering the case when $F(x)$ is not connected:

Let $C(M)$ be the family of compact subsets of $M$ and define
$$\mathcal{D}(M)=\{K\in C(M): \exists \,\, x\in M \mbox{and} \,\, \exists \,\, n\in\Z\,\,\, \mbox{such that}\,\,  F^n(x)=K\}.$$
\begin{maintheorem} \label{htop>0}
Let $M$ be a compact connected manifold of %a
dimension greater than zero and
 $F:M\multimap M$ be a continuous positive expansive set-valued map with $F(x)$ closed for all $x\in M$.
If $\mathcal{D}(M)$ is closed then $h_{top}(F)>0$.
\end{maintheorem}

We also prove that RW-expansive set-valued maps have positive topological entropy 
{without adding any extra hypothesis}.

\begin{maintheorem} \label{htopRW}
Let $M$ be a compact manifold of positive dimension and
 $F:M\multimap M$ be a continuous RW-expansive set-valued map.
Then $h_{top}(F)>0$.
\end{maintheorem}

Then we prove a result, Theorem \ref{t-multi}, that allows us to
define Lyapunov exponents for set-valued maps, denoted by $\chi^\pm(x)$, and
prove the following result:

\begin{maintheorem}\label{Lyapexpo positivo}
Let $M$ be a compact manifold and $F:M\multimap M$ a set-valued map such that there is a non trivial continuum $\gamma\subset M$ such that $\chi^+(x)>0$ for all $x\in\gamma$. Then the topological entropy of $F$ is positive.
\end{maintheorem}

This paper is organized as follows.
In Section \ref{setvalued} we set up notations and announce results proved elsewhere.
 In Section \ref{s-expansive}
 we introduce a notion of expansiveness in the context of  continuous as well as only upper semi continuous set-valued maps, give examples, and compare our definition with the one given in \cite{Wi}.
  In Section \ref{s7}  we prove Theorems \ref{t-expansive-entropy-positive}, \ref{htop>0}  and  \ref{htopRW}.
 In Sections \ref{s-Invariant-measure} and \ref{s-entropia-Lyapunov-exponent}
we introduce   measurable set-valued maps, and  invariant measures, and prove an auxiliary result, Theorem \ref{t-multi}, to
define Lyapunov exponents for set-valued maps and prove Theorem \ref{Lyapexpo positivo}.

\section{Set-valued maps} \label{setvalued}
In this section, we formalize definitions, set up the notations,
announce results proved elsewhere and describe basic facts about set-valued maps.

\begin{Df}
Let $X$ and $Y$ be two Hausdorff topological spaces and assume that for every point $x\in X$ a nonempty
closed subset $F(x)$ of $Y$ is given;
in this case, we say that $F$ is a set-valued map from $X$ to $Y$ and write\footnote{We will use the notation $f:X\to Y$ when referring to a single-valued map.}
$F:X \multimap Y$.
More precisely a set-valued map $F:X \multimap Y$ can be defined as a (closed) subset\footnote{We will require {for} every $x\in X$, $F(x)$ {to be} closed.}
$F\subset X\times Y$ such that the following condition is satisfied:
for all $x\in X$ there exists $y \in Y$ such that $(x,y) \in F$, i.e., $\forall\, x\in M: F(x)\neq\emptyset$.
\end{Df}

A set-valued map $F$ as defined above may be thought of as a single-valued map from $X$ to $2^Y$ or from $X$ to $C(Y)$ where $C(Y)$ is the family of closed subsets of $Y$. But, as is pointed out in \cite[page 6]{AF}, ``when we regard a set-valued map as a single-valued map from one set to the power set of the other
(supplied with any one of the topologies we can think of), we arrive
at continuity concepts which are stronger than both lower and
upper semicontinuity, introducing parasitic artifacts''.

\medbreak

If $F:X\multimap Y$ and $G:Y\multimap Z$ are two set-valued
maps, then the composition $G\circ F:X\multimap Z$ of $F$ and $G$ is defined for every $x\in X$ by:
$$(G\circ F)(x)=\bigcup \big\{G(y)\,:\, y\in F(x)\big\}\, .$$

We define for $B\subset Y$ the set $F^{+}(B)=\{x\in X\,:\, F(x)\subset B\}$ and $F^{-}(B)=\{x\in X\,:\, F(x)\cap B\neq\emptyset\}$.
In particular $F^{-}(y)=F^{-}(\{y\})=\{x\in X\,:\, F(x)\ni y\}$ is a natural candidate to be the inverse $F^{-1}$ in the case of set-valued maps and we shall write $F^{-1}(C)=F^{-}(C)$, $F^{-2}(C)=F^{-}\circ F^{-1}(C)$ and so on.
Observe that in the case of $f:X\to Y$ single-valued $F^{-}(y)=F^{+}(y)=f^{-1}(y)$ for every $y\in Y$.
The set $F^{+}(B)$ is called the small counter-image of $B$ and $F^{-}(B)$ the large
counter-image of $B$. Clearly $F^{+}(B)\subset F^{-}(B)$.

\begin{Df}
 A set-valued map $F:X\multimap Y$ is {\em upper semicontinuous}
(u.s.c.) provided for every open $U\subset Y$ the set $F^{+}(U)$ is open in $X$.
\end{Df}

Next, we collect some results from \cite{Gor} that we will use in the remaining { of this} paper.

\begin{Prop}
If $F:X\multimap Y$ is u.s.c. then the graph $\Gamma_F=\{(x,y)\in X\times Y\,:\, y\in F(x)\}$ is a closed subset of $X\times Y$.
\end{Prop}
\begin{proof}
See \cite[Proposition 14.4]{Gor}.
\end{proof}

\begin{Prop}
Let $F:X\multimap Y$ and $G:Y\multimap Z$ be u.s.c.. Then $G\circ F:X\multimap Z$ is upper semicontinuous too.
\end{Prop}
\begin{proof}
See \cite[Proposition 14.8]{Gor}.
\end{proof}

\begin{Df}
A set-valued map $F:X\multimap Y$ is {\em lower semicontinuous}
(l.s.c.) provided for every open $U\subset Y$ the set $F^{-}(U)$ is open in $X$.
\end{Df}

\begin{Prop}
A map $F:X\multimap Y$ is l.s.c. if and only if for every closed
set $C \subset Y$ the set $F^{+}(C)$ is a closed subset of $X$.

A map $F:X\multimap Y$ is u.s.c. if and only if for every closed
set $C \subset Y$ the set $F^{-}(C)$ is a closed subset of $X$.

\end{Prop}
\begin{proof}
See \cite[Proposition 15.3]{Gor}.
\end{proof}

A set-valued map $F$  is {\em continuous} if it is both lower and upper semicontinuous.

\noindent Given a subset $A\neq\emptyset$ of the metric space $(X,\dist)$ let $B(A,\varepsilon)=\{x\in X\,:\, \dist(x,A)\leq\varepsilon\}$.
Recall that the Hausdorff distance between compact subsets $A$ and $C$ of
$(X,\dist)$ is defined as
$$\dist_H(A,C)=\min\{\varepsilon>0\,:\, A\subset B(C, \varepsilon)\;\mbox{ and } C\subset B(A,\varepsilon)\}\, .$$

Let $(X,\dist)$ be a complete metric space and $(Y,d)$ a compact metric space. Let $C(Y)$ be the space of non-empty closed (hence compact) subsets of $Y$.

\begin{Df}
Let $F:X\multimap Y$ such that $X,Y$ are complete metric spaces and $x \in X$ such that $F(x)$ is closed on $Y$. We say that $F$ is Hausdorff-continuous iff given $\epsilon>0$ there is $\delta>0$ such that if $\dist(x,y)<\delta$ then $\dist_H(F(x),F(y))<\epsilon$.
\end{Df}

\begin{Prop} \label{cont=Hcont}
Let $M$ be a compact metric space and $F:M\multimap M$ an onto set-valued map. Then $F$ is Hausdorff-continuous if and only if $F$ is continuous.
\end{Prop}
\begin{proof}
Suppose first that $F$ is Hausdorff-continuous and let $V$ be an open subset of $M$. Let us prove that $F^{+}(V)$ is open.
Let $x\in F^{+}(V)$ then $F(x)\subset V$. Since $F(x)$ is closed and $M$ is compact $F(x)$ is compact. Since $V$ is open by compactness there is $\epsilon>0$ such that $B(F(x),\epsilon)\subset V$.
Since $F$ is Hausdorff-continuous there is $\delta>0$ such that if $\dist(x,x')<\delta$ then $\dist_H(F(x),F(x'))<\epsilon$. Hence $F(x')\subset B(F(x),\epsilon)\subset V$ and $F$ is u.s.c. .
Let us prove that $F$ is l.s.c., i.e., $F^{-}(V)$ is open. Let $x\in F^{-}(V)$ then $F(x)\cap V\neq \emptyset$. If $y\in F(x)\cap V$, since $V$ is open there is $\epsilon>0$ such that
$B(y,\epsilon)\subset V$. For such a $\epsilon$ since $F$ is Hausdorff-continuous we may find $\delta>0$ such that if $\dist(x,x')<\delta$ then $\dist_H(F(x),F(x'))<\epsilon$. Therefore $F(x)\subset B(F(x'),\epsilon)$ and by the  compactness of $F(x')$,
there is a point $y'\in F(x')$ such that $\dist(y,y')<\epsilon$. Hence $y'\in B(y,\epsilon)$ so that $F(x')\cap V\neq\emptyset$ proves that $F^{-}(V)$ is open.

Suppose now that $F$ is simultaneously u.s.c. and l.s.c. and let us prove that $F$ is Hausdorff-continuous.
Let $\epsilon>0$ and $x\in M$. Let $V=B(F(x),\epsilon)$ Then $F^{+}(V)$ is open and $F^{-}(V)$ is open too since $F$
is simultaneously u.s.c. and l.s.c. .
Let $U=F^{+}(V)\cap F^{-}(V)=F^{+}(V)$ then $U$ is open and clearly $x\in U$ and if $x'\in U$ then $F(x')\subset B(F(x),\epsilon)$.
To prove Hausdorff-continuity of $F$ we need to prove also that we have that
there is $\delta>0$ such that if $x'\in B(x,\delta)$ then $F(x)\subset B(F(x'),\epsilon)$. Since $F(x)$ is compact we may cover it by a finite number
of balls $B(y_j,\epsilon/2)$, $j=1,\ldots, n$ in such a way that if $y\in F(x)$ there is $y_j$ such that $\dist(y_j,y)<\epsilon/2$.
Since $F$ is l.s.c. there is $\delta_j>0$ such that if $x'\in B(x,\delta_j)$ then $F(x')\cap B(y_j,\epsilon/2)\neq\emptyset$. Let
$\delta=\min\{\delta_1,\ldots, \delta_n\}>0$. If $x'\in B(x,\delta)$ then $F(x')\cap B(y_j,\epsilon/2)\neq\emptyset$ for all $j=1,\ldots, n$.
Hence if $y\in F(x)$ there is $y_j\in F(x)$ such that $\dist(y,y_j)<\epsilon/2$ and there is $y'\in F(x')$ such that $\dist(y',y_j)<\epsilon/2$.
Therefore whenever $x'\in B(x,\delta)$ for all $y\in F(x)$ there is $y'\in F(x')$ such that $\dist(y,y')<\epsilon$. Thus $F(x)\in B(F(x'),\epsilon)$ and
$F$ is Hausdorff-continuous.
This finishes the proof.
\end{proof}

Because of Proposition \ref{cont=Hcont}, from now on we will say $F$ is continuous avoiding to say Hausdorff-continuous.

\section{Expansive multivalued maps}\label{s-expansive}
In this section, we introduce a notion of expansiveness in the context of upper semi continuous set-valued maps. %which  is valid for continuous ones.

\begin{Df} \label{suborbit}
 Let $(X,d)$ be a compact metric space and $F:X\multimap X$ an u.s.c. set-valued map.
A $F$-suborbit of $x$ is a set of the form
$$\{x_j\, : \, x_0=x, x_{j+1}\in F(x_j)\mbox{ for each } j\in\Z\}.$$
\end{Df}
Richard Williams \cite{Wi} proposed the following definition of expansiveness that we shall denote RW-expansiveness.

\begin{Df}\label{expRW}
A set-valued map $F: X \multimap X$ is RW-expansive on $X$ with expansive constant $\delta>0$ if $x,y\in X$, $x\neq y$ implies that for each $F$-suborbit $A$ of $x$ and  each $F$-suborbit $B$
of $y$, there exist $x_n\in A$, $y_n\in B$ such that $d(x_n,y_n)>\delta$.
\end{Df}

  Observe that according to this definition, every pair of distinct $F$-suborbits have to separate at a certain time $n\in \Z$
  and  if $0<\dist(x,y)<\delta$ then either $F(x)\cap F(y)=\emptyset$ or $F^{-1}(x)\cap F^{-1}(y)=\emptyset$, a sort of local injectivity of $F$ or of $F^{-1}$.
Indeed,  if $\dist(x,y)<\delta$ and $F(x)\cap F(y)\neq\emptyset$ and $F^{-1}(x)\cap F^{-1}(y)\neq\emptyset$ then we may find $z\in F^{-1}(x)\cap F^{-1}(y), $
  $w\in F(x)\cap F(y)$ and there would exist $F$-suborbits $\{\ldots, z_{-2},z_{-1}=z,x,w=w_1,w_2,\ldots\}$ and $\{\ldots, z_{-2},z_{-1}=z,y,w=w_1,w_2,\ldots\}$ which contradicts RW-expansiveness.

\vspace{0,2cm}

In the following, we propose a new definition of expansiveness for set-valued maps taking into account the separation of $F^n(x)$ and $F^n(y)$ in the Hausdorff distance.
To be more precise, let $M$ be a compact
metric space equipped with the distance $\dist:M\times M\to \R^+$ and $\dist_H$ the corresponding Hausdorff distance between compact subsets of $M$.

We will define the inverse of $F$ by the large counter image $F^{-1}(y)=F^{-}(y)=F^{-}(\{y\})=\{x\in X\,:\, F(x)\ni y\}$,% that we shall call $F^{-1}(y)$,
 and we also set $F^{-n}(y)=(F^{-1})^n(y)$, $n\geq 0$.

\begin{Df} \label{expan}
A continuous and surjective set-valued map $F:M\multimap M$ such that $F(x)$ is closed for all $x\in M$ is {\em expansive} if there is $\alpha>0$ such that if $x,y\in M$, $x\neq y$,
then $\dist_H(F^n(x),F^n(y))>\alpha $ for some $n\in\Z$ {or $F^n(x)=F^n(y)$ for every $n\in\Z\backslash\{0\}$}. If we require that if $x\neq y$ then $\dist_H(F^n(x),F^n(y))>\alpha $ for some $n\in\N$ {or $F^n(x)=F^n(y)$ for every $n\geq 1$} we say that $F$ is {\em positively expansive}.

\end{Df}

The next result guarantees that this definition is well-posed.

\begin{Prop} \label{compactos}
Let $C(M)$ be the family of compact subsets of $M$ and $F:M\multimap M$ an upper semi-continuous set-valued map with $F(x)\in C(M)$.
If $K\in C(M)$ then $F(K)\in C(M)$ and $F^{-1}(K)\in C(M)$.
\end{Prop}
\begin{proof}
Let $\{V_j,\, j\in J\}$ be an open covering of $F(K)$, where $J$ is a family of indexes. Hence, since $F(x)$ is compact, for every $x\in K$ there is a finite subcover $V_{j_1}(x),\ldots, V_{J_{n_x}}(x)$ covering $F(x)$. Let $U_x=\bigcup_{j_x=j_1}^{j_x=j_{n_x}} V_{j_x}$, hence $U_x$ is an open set covering $F(x)$, $x\in K$. Thus $\bigcup_{x\in K}U_x$ covers $F(K)$. Since $F$ is  upper semi continuous $F^-(U_x)$ is open and $\bigcup_{x}F^{-}(U_x)$ is an open covering of $K$.
By compactness of $K$ there is a
finite subcover of $K$ from $\bigcup_{x\in K}F^{-}(U_x)$, say, $\{F^{-}(U_{x_1}), \ldots, F^{-}(U_{x_n})\}$. Thus $\{U_{x_1},\ldots, U_{x_n}\}$ covers $F(K)$.
Since every $U_{x_i}, \, i=1,\ldots, n$ is the union of a finite number of sets $V_j$ we conclude that $F(K)$ is covered by a finite family from $V_j,\, j\in J$. Hence $F(K)$ is compact finishing the proof for $F$.

Now let us show that $F^{-1}(K)\in C(M)$. It suffices to show that $F^{-1}(K)$ is closed. Let $\{y_j\}_{j\in\N}\subset F^{-1}(K)$ and assume $y_j\to y$ when $j\to +\infty$.
It holds that $y_j\in F^{-1}(K)$ iff $F(y_j)\cap K \neq \emptyset$. Let $z_j\in F(y_j)\cap K$ and consider a convergent subsequence $z_{\ell_j}$, say to $z$. So $z_{\ell_j}\to z$ and since $K$ is compact we have $z\in K$. Therefore $F^{-1}(K)\supset F^{-1}(z)\ni y$ finishing the proof.
\end{proof}

\subsection{Examples.}

Let us give some examples of set-valued  expansive maps in the different senses.
\begin{itemize}
\item
Let $M=[0,1]$ and for $z\in [0,1]$ define $F(z)=\{2z\mod (1),(2z+1/2)\mod(1)\}$. Then $F^2(z)=\{4z\mod(1),(4z +1/2)\mod(1)\}$, etc., and $F$ is both expansive and RW-expansive.
(Recall that there are no expansive single valued maps on $[0,1]$, see \cite{Wi2}.

\item
Let $M=S^1$ and for $z\in S^1$ define $F(z)=\{z^2,-z^2\}$. Then $F^2(z)=\{z^4,-z^4\}$, etc. and $F$ is expansive and RW-expansive. We remark that there are no expansive homeomorphisms on $S^1$, see \cite{JU}.

\item
 Let $M=T^2$ that we identify with $[0,1]^2 \mod 1$ and let $F:T^2\multimap T^2$ be defined as $$(x,y)\multimap \left\{\begin{array}{l}
         (2x\mod 1,\quad y/2)   \\
                  (2x\mod 1,\; (1+y)/2)
                    \end{array}\right.\, .$$
  It is easy to see that $F$ is expansive with an expansivity constant $\alpha=1/4$.                                                                                                              Indeed, if ($x,y),(x',y')\in T^2$ and $x\neq x'$ then for some positive iterate $n$ we have $|2^n(x-x')|>1/4$ from which the result follows.
   If $x=x'$ but $y\neq y'$ then for some negative $-n$ we will have that $|2^n(y-y')|>1/4$, since
    $F^{-1}$ is given by
   $$(x,y)\multimap \left\{\begin{array}{l}
   (x/2\mod 1,\quad 2y\mod (1))   \\
   ((1+x)/2,\; 2y\mod (1))
         \end{array}\right.$$.

\item To describe the next example,
recall that if $f:M\to M$ is a homeomorphism,  $ x\in M$ and $\epsilon >0$, we  set
 $\Gamma_\epsilon(x)=\{y\in M\,:\, \dist(f^n(x),f^n(y))\leq\epsilon \}$.
An $N$-expansive homeomorphism $f$ is a homeomorphism such that there is $\epsilon>0$ such that for every $x\in M$ it holds that $\sharp \Gamma_\epsilon(x)\leq N$. Here $\sharp A$ denotes the cardinality of the set $A$.
Clearly $1$-expansive homeomorphisms are expansive homeomorphisms.

At \cite[Section 6]{APV} the authors exhibit a diffeomorphism $f: S\to S$, where $S$ is the bi-torus, which is $2-$expansive but it is not expansive \cite[Proposition 6.1]{APV}.
The non-wandering set $\Omega(f)$ consists of an expanding attractor $A$ and a contracting repeller $R$. Let $\alpha>0$ be a constant of $2$-expansiveness of $f$ and let $d=\dist(A,R)$. Note that  $d>0$. Fix $\beta=\min\{\alpha/2,d\}$
and given $x\in S$ define $F:S\multimap S$ by $F(x)=\cup_{ y\in\Gamma_{\beta}(x)} f(y)$. Then $F$ is an expansive multivalued map. Indeed, it is proved in \cite{APV} that for points in $\Omega(f)$ it holds that $\Gamma_\alpha(x)=\{x\}$ (and consequently $\Gamma_\beta(x)=\{x\}$), i.e., $f$ is expansive when restricted to $\Omega(f)$. If now $x\notin \Omega(f)$ then $\Gamma_\epsilon(x)$ may consist of two points, say $\Gamma_\alpha(x)=\{x,y\}$.
In this case $y\notin \Omega(f)$ and hence both $x,y$ are wandering points whose $\omega$- and $\alpha$-limit sets coincide.
If $z\notin\{x,y\}$ and $z\in R$ then by forward iteration by $f$, $\dist_H(F^n(z),F^n(x))$ approaches $d\geq \beta$. Similarly, if $z\in A$ by backward iteration by $f$, $\dist_H(F^{-n}(z),F^{-n}(x))$ approach $d\geq \beta$, $n>0$.
Finally if $z\notin \Omega(f)$ and $z\notin\Gamma_\beta(x)$ then since $f$ is 2-expansive there is $n\in\Z$ such that  the distances $\dist(f^n(x),f^n(z))>\beta$ and $\dist(f^n(y),f^n(z))>\beta$ too. Otherwise $z\in\Gamma_\alpha(x)$. This is clear if $\dist(f^n(x),f^n(z))\leq\beta$ but if $\dist(f^n(y),f^n(z)\leq \beta$ then {\it a fortiori} $z\in\Gamma_\alpha(x)$ and so it should be $x$ or $y$. We %see
{have shown} that $\beta$ is a constant of expansiveness for $F:M\multimap M$.
 This example is not RW-expansive. Indeed, as shown in \cite{APV}, there are points $x,y$, $y\in \Gamma_\alpha(x)$  arbitrarily close to each other and so $\forall \epsilon>0$ we can build sequences $\ldots, f^{-1}(x),x,f(x),f^2(x) ,\ldots$
 and $\ldots,f^{-1}(x),x,f(y),f^2(x),\ldots$ which are both in the $F$-orbit of $x$ and that only differs in the first iterate by $f$. Choosing $x,y$ so near that $\dist(f(x),f(y))<\epsilon$ we see that this example is not RW-expansive.
\item
Let $(M,\dist)$ be a compact metric space supporting an expansive homeomorphism $f:M\to M$ with an expansivity constant $\alpha>0$.
For each $g$ in a %certain
{suitable} $C^0$-neighborhood $\mathcal{N}(f)$ of $f$ we will define an expansive set-valued map $F_g:M\multimap M$.
%Recall that
It is known that the property of being expansive is not an open property in the $C^0$ topology.
Nevertheless, Lemma \ref{Lewowicz1} below states that a ``shadow`` of expansiveness rests.

\begin{Lem} \label{Lewowicz1}
Let $(M,\dist)$ be a compact metric space and $f:M\to M$ an expansive homeomorphism with a constant of expansivity $\alpha>0$.
Then for every $\delta>0$, there is a neighborhood $\mathcal{N}(f)$ of $f$ in the $C^0$-topology such that if $g \in \mathcal{N}(f)$ and for a
pair of points $x,y$ we have that $\forall n\in\Z:\dist(g^n(x),g^n(y))\leq \alpha$, then  for every $n\in\Z$ it holds that $\dist(g^n(x),g^n(y))\leq \delta$.
\end{Lem}

\begin{proof}
See \cite[Lemma 1.1]{Le}.
\end{proof}

That is to say, if $g\in\mathcal{N}(f)$, given two $g$-orbits $\cO(x)$ and $\cO(y)$ either  separate more than $\alpha$ at a certain time $n\in\Z$ or they rest $\delta$-close between them for all $n\in\Z$.

Now choose $\delta<\alpha/4$ and $g\in\mathcal{N}(f)$, where $\mathcal{N}(f)$ is the neighborhood corresponding to $\delta$
given by Lemma \ref{Lewowicz1}. For such a $\delta$ we consider in $M$, as in \cite{Le}, the subsets
   $A(x)=\{y\in M\,:\, \dist(g^n(x),g^n(y))\leq\alpha \,\,\mbox{for all}\,\,\  n\in\Z\}$ which define an equivalence relation that gives a partition of $M$.
   Indeed reflexivity and symmetry are trivial. Transitivity follows from the fact that
   if $y\in A(x)$ and $z\in A(y)$, we have first that  $\dist(g^n(x),g^n(y))\leq\alpha \,\,\mbox{for all}\,\,\  n\in\Z\}$ and
   $\dist(g^n(z),g^n(y))\leq\alpha \,\,\mbox{for all}\,\,\  n\in\Z\}$ and
   by Lemma \ref{Lewowicz1} we get  $\dist(g^n(x),g^n(y))\leq\alpha/4 \,\,\mbox{for all}\,\,\  n\in\Z\}$ and
   $\dist(g^n(y),g^n(z))\leq\alpha/4 \,\,\mbox{for all}\,\,\  n\in\Z\}$.
   Thus, $ \dist(g^n(x), g^n(z)) \leq\alpha/2<\alpha$ for all $n\in\Z$ asserting that $z\in A(x)$ and that $A(x)=A(y)=A(z)$.

 Now  define $F_g:M\multimap M$ by $F_g(x)=\cup_{y\in A(x)} g(y)$. Then $F_g$ is an expansive set-valued map with $F_g^{-1}(x)=\cup_{y\in A(x)} g^{-1}(y)$.    Indeed given two points $x$, $z$ if $z\in A(x)$ then $F_g(x)=F_g(z)$. If $z\notin A(x)$ then there is $n\in\Z$ such that $\dist(g^n(x),g^n(z))>\alpha$. Since points in $A(z)$ rest at a distance less than $\alpha/4$ {between them} by the action of $g$ %for every $n\in\Z$
 (and similarly for points in $A(x)$) we conclude that $\dist_H(F^n_g(x),F^n_g(z))>\alpha/2$ and $F_g$ is a set-valued expansive map.
   $\square$

\end{itemize}

%\section{Positive topological entropy for set-valued maps} \label{s7}
\section{Expansiveness and positive topological entropy.} \label{s7}

In this section, we prove that RW-expansive as well as positive expansive $F:M\multimap M$ set-valued maps, see Definition \ref{expan}, defined on non trivial Peano spaces have topological entropy strictly positive.

Let $F:M\multimap M$ be an expansive set-valued map.
Define $\mathcal{O}(F)$ as the set of all the suborbits of $F$ (see Definition \ref{suborbit}) and $\mathcal{O}(F)(x)$ as the set of all the suborbits $\{x_n\}_{n\in\Z}$ of $F$ such that $x_0=x$.
Similarly, given $A\subset M$ we define $\mathcal{O}(F)(A)=\cup_{x\in A}\mathcal{O}(F)(x)$.
 Given a suborbit $(x_k)_{k\in\Z}$ we denote by $\mathcal{O}_n(x)=(x_k)_{0}^{n}$ the segment of suborbits such that $x_0=x$, $x_j\in F(x_{j-1})$ for every $j=1,2,\ldots ,n$. Let us also define $\mathcal{O}_n(F)(A)=\cup_{x\in A}\mathcal{O}_n(x)$.

Let $F:M\multimap M$ be
 a set-valued map, $ n\in\N$ and $\epsilon > 0$. Recall that given a finite $A \subset M$,  $\sharp(A)$ denotes  the cardinality of $A$.

\begin{Df}
A subset $A \subset \mathcal{O}_n(F)(M)$ is $(n, \epsilon)$-spanning if, for any $(x_1, \ldots, x_n)\subset \mathcal{O}_n(F)$,
there is a $(y_1, \ldots, y_n)\subset A$ such that $\dist(x_i, y_i) < \epsilon$ for all $i=1,\ldots , n$.
Let $r_n(\epsilon)$ denote the minimum cardinality of an $(n, \epsilon)$-spanning set.
\end{Df}

\begin{Df}
 A subset $A \subset \mathcal{O}_n(F)(M)$ is $(n, \epsilon)$-separated if for any
$(x_1, \ldots , x_n)$, $(y_1, \ldots , y_n)$  in $A$,
$\dist(x_i, y_i)>\epsilon$ for at least one $i \in \{1, \ldots, n\}$. Let $s_n(\epsilon)$ denote the greatest cardinality
of an $(n, \epsilon)$-separated set.
\end{Df}

\begin{Df} \label{topentropy}
We define topological entropy for a set-valued function using $s_n(\epsilon)$ by
$$h_{top}(F) = \lim_{\epsilon\to 0}\limsup_{n\to\infty}\frac{1}{n}\log(s_n(\epsilon)).$$
\end{Df}
\begin{Obs}
\begin{enumerate}
\item
It is possible to define topological entropy using spanning sets as
$$h_{top}(F) = \lim_{\epsilon\to 0}\limsup_{n\to\infty}\frac{1}{n}\log(r_n(\epsilon)).$$
Indeed, at \cite{RT} it is proved that $r_n(\epsilon) \leq s_n(\epsilon) \leq r_n(\epsilon/2)$,  and this implies
 that for u.s.c. set-valued maps both definitions are equivalent.
\item
At \cite{CMM} the authors introduce definitions of topological entropy for set-valued maps that differ slightly from those given in \cite{RT}. In this article, we adopt the definition given in \cite{RT}.
\end{enumerate}
\end{Obs}

\begin{Teo} \label{conexos}
Let $F:M\multimap M$ be a continuous set-valued map such that $F(x)$ is compact and connected for every $x\in M$. If $\gamma$ is a continuum contained in $M$ then $F(\gamma)$ is compact and connected.
\end{Teo}
\begin{proof}
The compactness of  $F(\gamma)$ follows from Proposition \ref{compactos}.
  Let us prove that $F(\gamma)$ is connected.
 The proof goes by contradiction. Assume that there are non-empty disjoint closed subsets $K_1$, $K_2$, such that $F(\gamma)\subset K_1\cup K_2$.
  By compactness there are open sets $U_1\supset K_1$ and $U_2\supset K_2$ such that $\overline{U_1}\cap\overline{U_2}=\emptyset$.
Let $x\in\gamma$, since $F(x)=F(\{x\})$ is connected it must be contained either in $K_1$ or in $K_2$. Moreover, since $F$ is continuous there is $\delta>0$ such that if $\dist(x,x')<\delta$ then $F(x')\subset U_i$ if $F(x)\subset K_i$, $i=1,2$.
Therefore, again by the connectedness of $F(x')$, we get that $F(x')\subset K_i$.
Since we are assuming that $F(\gamma)$ is not connected, the subsets $A_i$ of $\gamma$ such that $x\in A_i$ if $F(x)\subset K_i$ are non void  and the previous argument implies that $A_i, 1 \leq i \leq 2,$ are open relative to $\gamma$.
As $F(x)$ is connected to all $x\in M$,  the union $A_1 \cup A_2$ is equal to $\gamma$,
 contradicting the connectedness of $\gamma$.
\end{proof}

\noindent {\bf{Proof of Theorem \ref{t-expansive-entropy-positive}.}}\/
Let $K\subset M$ be a non trivial continuum of diameter $\beta>0$ less than $\alpha$. Then there are points $x,y\in K$ such that $\dist(x,y)=\beta$.
Since $F$ is expansive there is $n\in\Z$ such that $\dist_H(F^n(x),F^n(y))>\alpha$ and hence $F^n(K)$ has a diameter greater than $\alpha$. Applying Theorem \ref{conexos}, we obtain that  $F$ is continuum wise expansive as a set-valued map.
Thus, as in \cite{CP}, we conclude that $h_{top}(F)>0$.
$\square$

\medbreak

{\em Question}: If $F(x)$ is not connected but $F:M\multimap M$ is expansive and  $\mbox{dim}(M)>0$, is it true that
the topological entropy of $F$ is positive?
Note that the subset $F(x)$ can be  finite  or countable, a Cantor set, etc. So we cannot expect to have that if $\gamma$ is connected then $F(\gamma)$ is connected without imposing any other condition.
For instance, if $f:S^1\to S^1$ is a Denjoy  homeomorphism defined on the unit circle  $S^1$, the restriction of $f$ to $\Omega(f)$
is a Cantor set, that $f_{\Omega(f)}:\Omega(f) \to \Omega(f)$ is expansive and
$h_{top}(f_{\Omega(f)})=0$ (every homeomorphism defined on $S^1$ has zero topological entropy).
\medbreak

Below we give proof that if $M$ is a compact manifold and $F:M\multimap M$ is continuous and satisfies some additional hypothesis then $h_{top}(F)>0$. To do so, we first prove some auxiliary lemmas.

\begin{Lem} \label{4.4}
Let $M$ be a compact manifold and $F:M\multimap M$ be positive expansive and continuous with an expansivity constant $\alpha>0$.
Let $\delta>0$ and $x,y\in M$ such that $\dist(x,y)\geq \delta$. Then there is $n_0>0$ such that $\dist_H(F^j(x),F^j(y))\geq \alpha$ for some $j\in [0,n_0]$.
\end{Lem}
\begin{proof}
The proof goes by contradiction.
Assume there is a pair of sequences $\{x_n\}_{n\in\N}$ and $\{y_n\}_{n\in \N}$ in $M$
such that
$\dist(x_n,y_n)\geq \delta$ and
$\dist_H(F^j(x_n),F^j(y_n))< \alpha$  for every $j\in [0,n]$. By compactness of $M$ there are subsequences $\{x_{n_k}\}$ and $\{y_{n_k}\}$ from $\{x_n\}$ and $\{y_n\}$  converging to $x$ and $y$ respectively such that $\dist_H(F^j(x),F^j(y))\leq \alpha$ for every $j\in\N$ and $\dist(x,y)\geq\delta$.
  Thus contradicting expansiveness.
\end{proof}

\begin{Lem} \label{4.5}
Let $M$ be a compact manifold and $F:M\multimap M$ be positive expansive and continuous with $\alpha>0$ an expansivity constant. Let $x,y\in M$ and $n_1\in\N$ be such that $\dist_H(F^j(x),F^j(y))=\alpha/4$ for some $j\in\N$ such that $j\leq n_1$. Then there is $n_2\in\N$ such that $\dist_H(F^{j+k}(x), F^{j+k}(y))\geq \alpha$ for some $k\in\N$ such that $k\leq n_2$.
\end{Lem}
\begin{proof}
For if it were not true there would exist $x_n,y_n$ such that  $\dist_H(F^{j_n}(x_n),F^{j_n}(y_n))=\alpha/4$  for some $j_n\in\N$ such that $j_n\leq n_1$ and $\dist_H(F^{j_n+k}(x_n), F^{j_n+k}(y_n))< \alpha$ for every $k\leq n$. Since the set $\{0,1,\ldots,n_1\}$ is finite and
$j_n\in \{0,1,\ldots,n_1\}$ there is some value $j_n$ that repeats infinitely many times as $n\to\infty$. Let us denote this  $j_n$ by $j_0$.
Without loss of generality, we may assume that $j_0$ is the same for all the sequences $\{(x_n,y_n)\}\in M\times M$.
By compactness of $M$, (and consequently of $M\times M$) there is an accumulation point $(x,y)$ of $(x_n,y_n)$.
Again without loss of generality, we can assume that $(x_n,y_n)\to (x,y)$. Since $F$ is Hausdorff continuous and $\dist_H(F^{j_0}(x),F^{j_0}(y))=\alpha/4$ we have $x\neq y$.
We also have that $\dist_H(F^{j_0+k}(x_n), F^{j_0+k}(y_n))< \alpha$  for every $k\in [0,n]$ and since $F$ is Hausdorff continuous, we get
that $\dist_H(F^{j_0+k}(x), F^{j_0+k}(y))\leq \alpha$ for all $k\in\N$ which contradicts that $F$ is positively expansive.
Therefore there exists $n_2>0$ verifying the thesis.
\end{proof}

Observe that $n_2$ given by Lemma \ref{4.5} depends not only on $\alpha/4$ but also on the number $j$ of iterates by $F$ needed to have that $\dist_H(F^j(x),F^j(y))=\alpha/4$.
To have a uniform bound not depending on $j$ we need to add an extra hypothesis on $F$.
To do so, we start defining the set $\mathcal{D}(M)$ as
$$\mathcal{D}(M)=\{K\in C(M): \exists \,\, x\in M \mbox{and} \,\, \exists \,\, n\in\Z\,\,\, \mbox{such that}\,\,  F^n(x)=K\}.$$

\begin{Lem} \label{4.6}
Let $\mathcal{D}(M)$ be as above and assume it is a closed subset of $C(M)$. Then the bound $n_2$ of Lemma \ref{4.5} depends only on the quantity $\alpha/4$.
\end{Lem}
\begin{proof}
By hypothesis $\mathcal{D}(M)$ is closed and hence $\mathcal{D}(M)\times \mathcal{D}(M)$ is closed on the compact metric space $C(M)\times C(M)$.
Therefore $\Delta_\mathcal{D}$ the subset of $ \mathcal{D}(M)\times\mathcal{D}(M)$  of closed subsets of $M\times M$ of the form $(F^n(x),F^n(y))$ for some $x,y\in M$ and $n\in \N$ is also closed on $C(M)\times C(M)$.
Assuming that the bound $n_2$ is not uniform we may find a sequence of pairs $(x_n,y_n)\in M\times M$ such that $\dist_H(F^{k_n}(x_n),F^{k_n}(y_n))=\alpha/4$ and
$\dist_H(F^{k_n+j}(x_n),F^{k_n+j}(y_n))<\alpha$ for every $j\in[0,n]$.
Since $\Delta_\mathcal{D}$ is closed the Hausdorff limit of $(F^{k_n}(x_n),F^{k_n}(y_n))$
is a point of $\Delta_\mathcal{D}$ and so there are $j\in \N$ and points $x,y\in M$ in the limit set of $(x_n,y_n)$ such that $\dist_H(F^j(x),F^j(y))=\alpha/4$ but $\dist_H(F^n(x),F^n(y))\leq \alpha$ for every $n\in\N$. Thus contradicting that $F$ is positively expansive.

\end{proof}

\noindent {\bf{ Proof of Theorem \ref{htop>0}.}}\/
Consider a geodesic arc $\gamma\subset M$ of length $\ell(\gamma)=\frac{\alpha}{4}$ image of the interval $[0,1]\subset \R$ by a function $\gamma$ (the use of $\gamma$ for the arc and the function simultaneously should not cause problems). Let $x_0=\gamma(0)$ and $x_1=\gamma(1)$
Since $F$ is positively expansive, by Lemma \ref{4.4}, there is $n_0\in\Z^+$, depending on $\alpha/4$, such that $\dist_H(F^{n_1}(x_0),F^{n_1}(x_1))\geq \alpha$ for some $n_1\in [0,n_0]$, $n_1>0$ such that $\dist_H(F^j(x_0),F^j(x_1)) < \alpha$ for every $j\in [0,n_1-1]$.
Next, we rename the points $x_0$ and $x_1$ calling them respectively by $x_{00}$ and $x_{11}$.
Since $F$ is continuous there are $x_{01},x_{10}\in \gamma$ such that $x_{00}<x_{01}<x_{10}<x_{11}$ (using implicitly the order of $[0,1]$) and  $\dist_H(F^{n_1}(x_{00}),F^{n_1}(x_{01})=\alpha/4$ and $\dist_H(F^{n_1}(x_{10}),F^{n_1}(x_{11})=\alpha/4$ and consequently $\dist_H(F^{n_1}(x_{01}),F^{n_1}(x_{10})\geq \alpha/2$.
By Lemmas \ref{4.5} and \ref{4.6} there exists $N\in\N$ and there are $n_2,n'_2\in [0,N]$ such that
$$\dist_H(F^{n_2}(F^{n_1}(x_{00})),F^{n_2}(F^{n_1}(x_{01})))\geq \alpha\quad  and$$
 $$\dist_H(F^{n'_2}(F^{n_1}(x_{10})),F^{n'_2}(F^{n_1}(x_{11})))\geq \alpha\, .$$
Again we rename the points $x_{00}$, $x_{01}$ $x_{10}$ and $x_{11}$ calling them respectively by $x_{000}$, $x_{011}$, $x_{100}$ and $x_{111}$.
We may repeat the argument and find points $x_{001}$, $x_{010}$, $x_{101}$ and $x_{110}$ on $\gamma$ such that
$x_{000}<x_{001}<x_{010}<x_{011}<x_{100}<x_{101}<x_{110}<x_{111}$ and
$$\dist_H(F^{n_2}(F^{n_1}(x_{000})),F^{n_2}(F^{n_1}(x_{001})))= \alpha/4\, ,$$ $$\dist_H(F^{n_2}(F^{n_1}(x_{010})),F^{n_2}(F^{n_1}(x_{011})))= \alpha/4\, ,$$
$$\dist_H(F^{n'_2}(F^{n_1}(x_{100})),F^{n'_2}(F^{n_1}(x_{101})))= \alpha/4\quad and$$
$$\dist_H(F^{n'_2}(F^{n_1}(x_{110})),F^{n'_2}(F^{n_1}(x_{111})))= \alpha/4\, .$$

Again there exist positive integers $n_3,n'_3,n''_3,n'''_3$ bounded by $N$ such that
   $$\dist_H(F^{n_3}(F^{n_2}(F^{n_1}(x_{000}))),F^{n_3}(F^{n_2}(F^{n_1}(x_{001}))))\geq \alpha,\; \ldots \ldots$$
    $$\ldots\; ,\dist_H(F^{n'''_3}(F^{n'_2}(F^{n_1}(x_{110}))),F^{n'''_3}(F^{n'_2}(F^{n_1}(x_{111}))))\geq \alpha\, .$$

Arguing recursively, at the $k^{th}$ step we find a sequence of $2^k$ points $x_{i_1,i_2,\ldots,i_k}$, where $i_j\in\{0,1\}$  for $j=1,2,\ldots,k$ which are $\alpha/4$ separated in the distance $\dist_{H,k}(x,y)=\max_{j\in[0,k]}\{\dist_H(F^j(x),F^j(y))\}$.

It is well known that this implies that $h_{top}(F)$ is positive.
$\square$

\medbreak

Let us now show that the topological entropy of RW-expansive set-valued maps defined on a compact manifold $M$ is positive.
To this end, we first prove that RW-expansiveness defined on compact metric spaces is uniform.
Recall that a space is called sequentially compact if every sequence of the space admits a convergent subsequence.

\begin{Lem} \label{seq compact}
Let $\{M_j\}_{j\in\N}$ be a family  of sequentially compact metric space. Then the countable product $\prod_{j\in\N}M_j$ is sequentially compact.
\end{Lem}
\begin{proof}
See \cite[Chapter 11, Theorem 1.12]{Jo}.
\end{proof}

\begin{Lem} \label{uniforme}
If $F$ is an RW-expansive set-valued map with a constant of expansiveness $\alpha>0$ then
for every $\delta>0$ there is $N=N(\delta)\in\N$ such that if $\dist(x,y)\geq \delta$ then for every pair of $F$-sequences $\{x_n\}$, $\{y_n\}$, $x_0=x$, $y_0=y$, $\dist(x_n,y_n)>\alpha$ for some $n\in [-N,N]$.
\end{Lem}
\begin{proof}
Arguing by contradiction assumes that the contrary holds and that there is a pair of sequences in $M$ $x^k, k\in \N$, and $y^k, k\in\N$, such that
$\dist(x^k,y^k)\geq \delta$ and $F$-sequences $\{x(k)_n\}_{n\in\Z}$ and $\{y(k)_n\}_{n\in\Z}$ with $x^k=x(k)_0$ and $y^k=y(k)_0$, such that  for every $n\in [-k,k]$ it holds $\dist(x(k)_n),y(k)_n))\leq \alpha$.
In particular, for $n=0$ we have $\dist(x(k)_0,y(k)_0)\leq \alpha$. On account of Lemma \ref{seq compact} there is a convergent subsequence of $x(k)_n$ and $y(k)_n$ for every $n\in\Z$, say $\{x({k_j})_n\}_{n\in \Z}$ and $\{y({k_j}_n\}_{n\in\Z}$.
Let $x(k_j)_n\to x_n$ and $y(k_j)_n\to y_n$ when $j\to +\infty$. Hence we have $\delta\leq \dist(x_0,y_0)\leq\alpha$ and for every $n\in\Z$, $\dist(x_n,y_n))\leq \alpha$ contradicting that $F$ is RW-expansive.
\end{proof}

\noindent {\bf{Proof of Theorem \ref{htopRW}.}}\/
Let $\alpha>0$ be a constant of RW-expansiveness for $F$.
Since a Peano space $M$ is connected and locally connected, as in Theorem \ref{htop>0}, we may find a geodesic arc $\gamma:[0,1]\to M$ joining two different points $x=\gamma(0)$ and $y=\gamma(1)$ with $\alpha/4=\dist(x,y)$. By Lemma \ref{uniforme} there is $N>0$ such that for every pair of $F$-sequences $\{x_n\}_{n\in \Z}$ and $\{y_n\}_{n\in\Z}$ with $x_0=x$ and $y_0=y$  there is  $n_1\in\Z$ with $|n_1|\leq N$ and  it holds that $\dist(x_{n_1},y_{n_1})>\alpha/2$. We may assume that $n_1>0$ and that for $0\leq j< n_1$ we have that $\dist(x_j,y_j)<\alpha/2$. Let $z=\gamma(1/2)$, and an $F$ sequence $\{z_n\}_{n\in\Z}$ with $z_0=z$ such that either $\dist(x_{n_1},z_{n_1})\geq\alpha/4$ or $\dist(y_{n_1},z_{n_1})\geq\alpha/4$. Moving the parameter $t\in[0,1]$ we may assume that we have both $\dist(x_{n_1},z_{n_1})\geq\alpha/4$ and $\dist(z_{n_1},y_{n_1})\geq\alpha/4$. Taking a value $0<t\leq 1$, on account of the continuity of $F$ and $\gamma$, we can suppose that $\dist(x_{n_1},y_{n_1})=\alpha/2$ and that $\dist(x_{n_1},z_{n_1})=\alpha/4$ and $\dist(z_{n_1},y_{n_1})=\alpha/4$.
Again by Lemma \ref{uniforme} there are $0<n_2,\, n'_2<N$ and $F$-sequences such that $\dist(x_{n_1+n_2},z_{n_1+n_2})\geq\alpha/2$ while $\dist(x_j,z_j)<\alpha/2$ for every $0\leq j<n_1+n_2$ and
$\dist(y_{n_1+n'_2},z_{n_1+n'_2})\geq\alpha/2$ while $\dist(y_j,z_j)<\alpha/2$ for every $0\leq j<n_1+n'_2$. Moreover, $n_1+n_2<2N$ and $n_1+n'_2<2N$.
 As the reader can see the rest of the proof is similar to that of Theorem \ref{htop>0}, the quantity of points at a distance greater $\alpha/2$ for iterates bounded by $kN$ is greater than $2^k$ implying that $h_{top}(F)>0$.
\qed

\section{Invariant measures and Lyapunov exponents for set-valued maps}\label{s-Invariant-measure}
The goal of  this section is to introduce measurable set-valued maps, and invariant measures and finally to
define the Lyapunov exponent for set-valued maps.
\medbreak

Let $M$ be a compact $d$-dimensional manifold and $\mathcal{B}(M)$ the family of Borelian subsets of $M$.
A set-valued map $F:M\multimap M$, with $F(x)$ compact $\forall x\in M$,
is  {\em{measurable}} if for every open subset $A\subset M$ we have that $F^{-1}(A)\in \mathcal{B}(M)$. Equivalently
$F$ is  measurable if for every closed subset $C\subset M$ we have that $F^{-1}(C)\in \mathcal{B}(M)$. Here, if $C \subset M, \,\, F^{-1}(C)$, is the largest counter-image of $C$.

Recall that  given a  measure $\mu$ %defined 
on $M$ and a single-valued measurable function $f:M\to M$, the measure $\mu$ is
$f$-invariant if for all Borelian subsets $B\subset M$ it holds that $\mu(f^{-1}(B))=\mu(B)$. This is no longer valid for set-valued maps.
The following notion of $F$-invariant measure, \cite{AFL,MA}, is adopted for u.s.c. set-valued maps.
\begin{Df}
Let $M$ be a compact manifold and
let $F:M\multimap M$ be an u.s.c. set-valued map. We say that $\mu$ is an invariant probability measure for $F$ if
given a Borel subset  $B\subset M$ it holds that $\mu(B)\leq \mu(F^{-1}(B))$.
In this case, the set-valued map $F$ defines a super stationary process, see \cite{Ab,Sch}.
\end{Df}

The existence of  an invariant measure for an u.s.c. set-valued map is guaranteed in our setting (see \cite[Theorem 8.9.4]{AF}, see also \cite{AFL,MA}).

Proposition \ref{compactos}  enables us to define the following subsets of $C(M)$:

For $n\geq 0$
 $$B^*_K(\delta,n)=\{A \in C(M)\setminus K\,:\, \dist_H(F^j(K),F^j(A))\leq\delta,\, \forall\, j=0,1,\ldots,n\}\, .$$
and for $n<0$
$$B^*_K(\delta,n)=\{A\in C(M)\setminus K\,:\, \dist_H(F^j(K),F^j(A))\leq\delta,\,
 \forall\, j=n,n+1,\ldots,-1,0\} .$$

\noindent For $n\in\Z$, $\delta>0$ and $x\in M$ let us define
\begin{equation}\label{Hdelta}
H_\delta(K,n)=\sup_{A\in B^*_K(\delta,n)} \left\{\frac{\dist_H(F^n(K),F^n(A))}{\dist_H(K,A)}\right\}
\end{equation}
and
\begin{equation}\label{hdelta}
h_\delta(K,n)=\inf_{A\in B^*_K(\delta,n)} \left\{\frac{\dist_H(F^n(K),F^n(A))}{\dist_H(K,A)}\right\}\, .
\end{equation}
Note that when $K=\{x\}$ and $A=\{y\}$, $\dist_H(\{x\},\{y\})=\dist(x,y)$.

\begin{Obs}
\begin{enumerate}
\item
If $K,A\subset C(M)$ and $K\neq A$ then there is a point $x\in K$ such that $x\notin A$ (so that $\dist(x, A)>0$)
or there is a point $y\notin K$ such that $y\in A$.
In any case
$\dist_H(K,A)>0$ and so $H_\delta(K,n)$ and $h_\delta(K,n)$ are well defined.
\item
For $x,y\in M,\, x\neq y$ it is possible to have that $F^n(x)=F^n(y)$ for some $n\in \N$.
And if this occurs then $F^{n+k}(x)=F^{n+k}(y)$ for every $k\geq 0$.
Similarly, if $F^{-n}(x)=F^{-n}(y)$ then $F^{-n-k}(x)=F^{-n-k}(y)$ for every $k\geq 0$.
 But if $F$ is expansive then either there is
 $n_0>0$ such that
$\dist_H(F^{n_0}(x),F^{n_0}(y))>\alpha$ or
there is $-n_1<0$ such that $\dist_H(F^{-n_1}(x),F^{-n_1}(y))>\alpha$.
 Thus either  for $0\leq j\leq n_0$
or for $-n_1\leq j\leq 0$, we cannot have $F^j(x)=F^j(y)$.
Moreover, if $F$ is continuous we also have that if $y\to x$ then the time $n$ needed to get that $\dist_H(F^n(x),F^n(y))>\alpha$ or  $\dist_H(F^{-n}(x),F^{-n}(y))>\alpha$ goes to infinity. Therefore, if  $\dist_H(\{x\},A)<\delta$, with $\delta>0$ small enough, we have
$\frac{\dist_H(F^n(\{x\}),F^n(A))}{\dist_H(\{x\},A)}>0$ and so
$H_\delta(x,n)=\sup_{A\in B^*_{\{x\}}(\delta,n)} \left\{\frac{\dist_H(F^n(\{x\}),F^n(A))}{\dist_H(\{x\},A)}\right\}>0$.
\item
Consider a pseudo Anosov diffeomorphism $f$ defined on a surface $S$ with %a 
genus greater or equal to 2.
Let $p$ be one singularity of $f$ and
 $x\in W^s_\epsilon(p)$.
 Now define $F:S\multimap S$ by $F(x)=Cc(B(f(x),r)\cap W^u_\epsilon(f(x)),f(x))$ where $r>0$ is small.
 Here $Cc(B(x,r)\cap X, x)$ is the connected component of $B(x,r)\cap X$ containing $x$.
Then $F$ is u.s.c. but cannot be continuous. Indeed, pick a point $x$ close to $p$ on one stable prong of $p$. Then no matter how near $p$ is {from} $x$
we have that  $\dist_H(F(x),F(p))>\epsilon/2$. See Figure \ref{f-unica}.

\begin{figure}[htb]
\begin{center}
\includegraphics[scale=0.38]{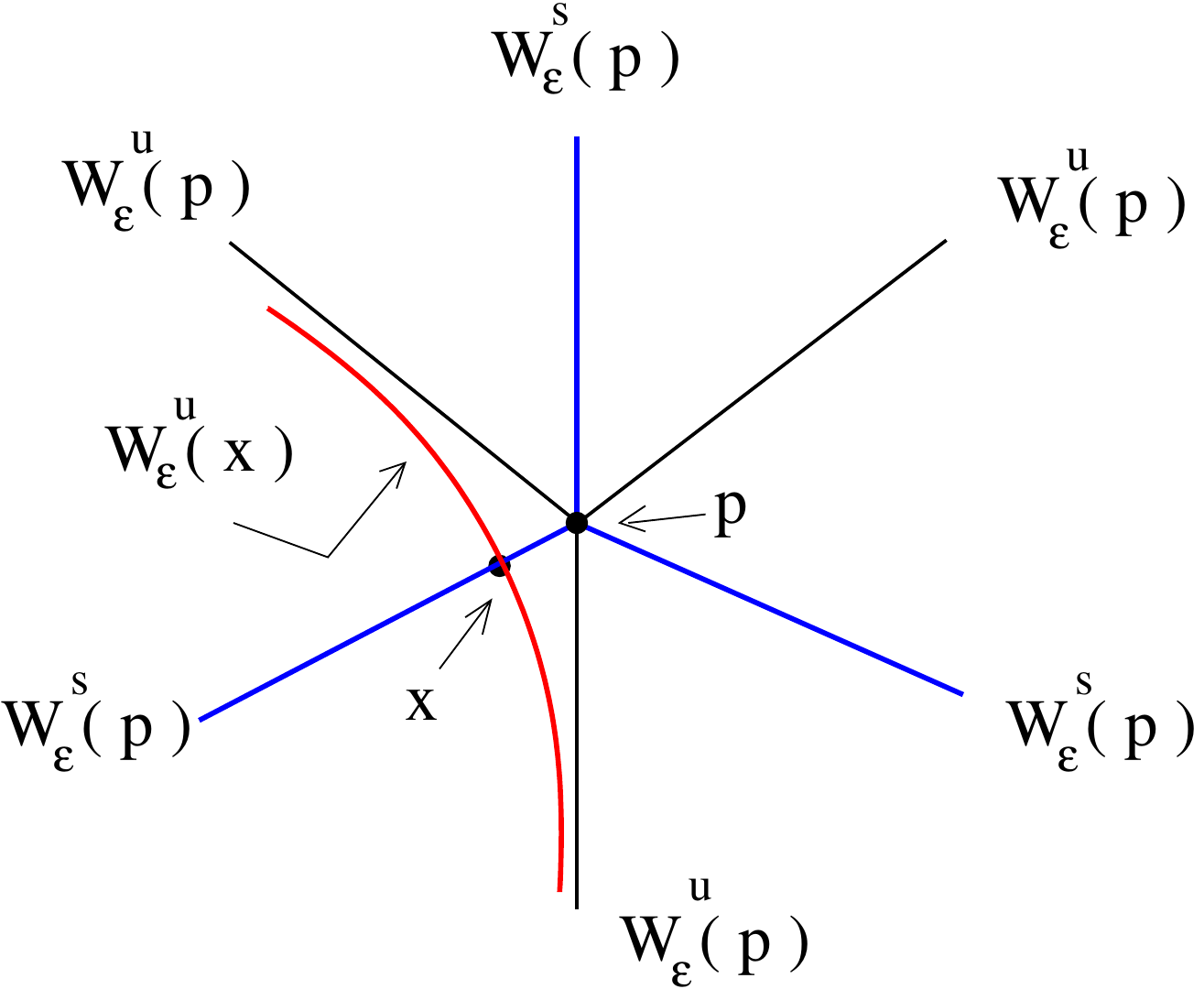}
\caption{$\dist_H(W^u_\epsilon(x),W^u_\epsilon(p))>\epsilon/2$}\label{f-unica}
\end{center}
\end{figure}

\end{enumerate}
\end{Obs}

\noindent {\em{ Notation.}}\/ For the case of the singleton $\{x\}$ we will write $x$ instead of the more cumbersome $\{x\}$.
We also will set $F^0(x)=\{x\}=x$ for every $x\in M$.
We will use the same notation (capital letters) to denote $F:C(M)\to C(M)$ the induced
(single-valued) map on compact subsets given by $F:M\multimap M$.
We think that this double use would not confuse.

\begin{Df}
 We say that $F:C(M)\to C(M)$ is Lipschitz continuous
  if there is a constant $k\geq 1$ such that  for every $K,A\in C(M)$
  it holds
  $$\dist_H(F(K),F(A))\leq k\,\dist_H(K,A)\,\,\mbox{ and }\,\,
 \dist_H(F^{-1}(K),F^{-1}(A))\leq k\,\dist_H(K,A).$$
 This in particular implies that $F:M\multimap M$ and $F^{-1}:M\multimap M$ are Lipschitz, i.e.,
 $$\dist_H(F(x),F(y))\leq k\, \dist_H(\{x\},\{y\})=k\,\dist(x,y)\,\,\mbox{and}$$
 $$dist_H(F^{-1}(x),F^{-1}(y))\leq k\, \dist_H(\{x\},\{y\})=k\,\dist(x,y).$$
\end{Df}

Using equation (\ref{Hdelta}) we define for $n>0$
$$\Lambda^+_\delta(x)=\limsup_{n\to+\infty}\frac{1}{n}\log (H_\delta(x,n)),$$
and  using equation (\ref{hdelta}) we define for $n<0$
$$\lambda^-_\delta(x)=-\limsup_{n\to-\infty}\frac{1}{n}\log (h_\delta(x,n))\, .$$

\begin{thm}\label{t-multi}
Let $M$ be a compact manifold and
 $F:M\to M$ a Lipschitz continuous set-valued map such that $F(x)\subset M$ is compact. Let $\mu$ be $F$ invariant.
Then for $x\in M$ $\mu$ a.e. it holds that
the limits

$$\Lambda^+_\delta(x)=\lim_{n\to+\infty}\frac{1}{n}\log (H_\delta(x,n))\, $$
 $$\lambda^-_\delta(x)=-\lim_{n\to-\infty}\frac{1}{n}\log (h_\delta(x,n))$$
do exist and are finite.
Moreover, $\Lambda^+_\delta(x)$ and $\lambda^-_\delta(x)$ are $F$-invariant $\mu$ a.e.
\end{thm}

\begin{proof}
By hypothesis,  $F$ is Lipschitz continuous and hence there is $k>0$ such that $\dist_H(F^{\pm 1}(K),F^{\pm 1}(A))\leq k\,\dist_H(K,A)$ for every $K,A\in C(M)$. Thus,  for all pairs of points $x,\, y \, \in M,\,\, x\neq y$, we have that
$$\frac{\dist_H(F(x),F(y))}{\dist(x,y)}\leq K .$$

\noindent
The last inequality implies that $\sup_{x,y\in M, x\neq y}\frac{\dist_H(F^n(x),F^n(y))}{\dist_H(x,y)}\leq K^{n}$ for all $n\in\N$,
and hence $\log(|H_\delta(x,n)|)\leq n\log(K)$ for all $\delta>0$, $x\in M$ and $n\in\N$. Therefore
\begin{equation} \label{ecu}
\sup_{n\in\N\backslash\{0\}}\frac{1}{n}\int_M|\log(H_\delta(x,n))|\mu(dx)<\infty\;.
\end{equation}

In case that $\dist_H(F^n(x),F^n(A))=0$ then $F^n(x)=F^n(A)$ and therefore  we have that
$\dist_H(F^{n+k}(x),F^{n+k}(A))=0$ too.
In this case, we set, by convention,
$$\left(\frac{\dist_H(F^{n+k}(x),F^{n+k}(A))}{\dist_H(F^n(x),F^n(A))}\right)=0.$$
This convention unifies the treatment given below.
It is easy to see that $B^*_{\{x\}}(\delta,n+k)\subset B^*_{\{x\}}(\delta,n)$ and  $B^*_{\{x\}}(\delta,n+k)\subset B^*_{\{F^n(x)\}}(\delta,k)$.
For, if $A$ is such that $\dist_H(F^j(x),F^j(A))\leq\delta,\, \forall\, j=0,1,\ldots,n+k$ then setting $B=F^n(A)$ we get
that $\dist_H(F^{n+j}(x),F^j(B))\leq\delta$, for every $ j=0,1,\ldots,k$. Thus $B^*_{\{x\}}(\delta,n+k)\subset B^*_{\{F^n(x)\}}(\delta,k)$.

\noindent
Hence
$$H_\delta(x,n+k)=\sup_{A\in B^*_{\{x\}}(\delta,n+k)} \{\dist_H(F^{n+k}(x),F^{n+k}(A))/\dist_H(\{x\},A)\}= $$
$$\sup_{A\in B^*_{\{x\}}(\delta,n+k)} \left\{\left(\frac{\dist_(F^{n}(x),F^{n}(A))}{\dist_H(\{x\},A)}\right) \left(\frac{\dist_H(F^{n+k}(x),F^{n+k}(A))}{\dist_H(F^n(x),F^n(A))}\right)\right\}\leq$$
$$\leq\sup_{A\in B^*_\{x\}(\delta,n)}\! \left\{\frac{\dist(F^n(x),F^n(A))}{\dist_H(\{x\},A)}\right\}\cdot \!\!\!\!\sup_{B\in B^*_{F^n(x)}(\delta,k)} \!\left\{\frac{\dist_H(F^{n+k}(x),F^k(B))}{\dist_H(F^n(x),B)}\right\}=$$
$$=H_\delta(x,n)\cdot H_\delta((F^n(x),k). $$
Now,  set $Y(\delta,x,n)=\log (H_\delta(x,n))$. Then we have that $Y(\delta,x,n+k)\leq Y(\delta,x,n)+Y(\delta,f^n(x),k)$, i.e., $Y(\delta,x,n)$ is subadditive. Moreover, since $\mu$ is $F$-invariant, $Y(\delta,x,n)$ is super stationary in the sense that $\mu(Y(\delta,F(x),n+1))\leq \mu(Y(\delta,x,n))$, $\mu$-a.e., \cite{MA,AFL}. Applying the generalization of Kingman's Subadditive Ergodic Theorem given at \cite{Ab} (see also \cite{Sch}),
 we conclude that the limit $\Lambda^+_\delta(x)=\lim_{n\to+\infty}\frac{1}{n}\log (H_\delta(x,n))$ exists $\mu$-a.e.. Equation (\ref{ecu}) guarantees that this limit is $<+\infty$.

 Similarly, as $h_\delta(x,n+k)\geq h_\delta(x,n)\cdot h_\delta(f^n(x),k)$ we get that $y(\delta,x,n)=-\log(h_\delta(x,n))$ is subadditive too which implies that the limit $\lambda^-_\delta(x)= - \lim_{n\to-\infty}\frac{1}{n}\log (h_\delta(x,n))$ exists.

 Since $F$ is $\mu$-invariant, the quantities $\Lambda^+_\delta(x)$ and $\Lambda^-_\delta(x)$ are also $F$-invariant.
 The proof of Theorem \ref{t-multi} is completed.

\end{proof}
Observing that $\Lambda^+_\delta(x)$ decreases and $\Lambda^-_\delta(x)$ increase when $\delta\to 0$ we define
\begin{Df}
The maximal and minimal Lyapunov exponents for $F$ at $x\in M$ are given by
$\chi^+(x)=\lim_{\delta\to 0}\Lambda^+_\delta(x)$ and $\chi^-(x)=\lim_{\delta\to 0}\Lambda^-_\delta(x)$ respectively.
\end{Df}

\medbreak

\section{Positive entropy for set-valued maps such that $\chi^+(x)>0$}\label{s-entropia-Lyapunov-exponent}

In this section, we prove that the topological entropy, as given by Definition \ref{topentropy}, of a set-valued map $F:M\multimap M$ is positive whenever
there is a non trivial continuum $\gamma\in M$ such that $\chi^+(x)>0$ for every $x\in\gamma$. %A similar result also holds if $\chi^-(x)<0$ reversing the time $n$. 
Let us emphasize that $\chi^+(x)>0$ for an isolated point $x\in M$ solely cannot ensure that $h_{top}>0$. For,  %on 
\cite[page 140]{Ka} exhibits a diffeomorphism $f$ of the two sphere such that $\chi^+(x)>0$ for all non-wandering points but $h_{top}(f)=0$. %The crucial point 
In that example %to make vanish the topological entropy is the fact that 
$\Omega(f)$ is a finite set.

\vspace{0,2cm}

\noindent {\bf{Proof of Theorem \ref{Lyapexpo positivo}.}}\/
Since $\gamma$ is compact and $\chi^+(x)>0$ for every $x\in\gamma$ there is a positive lower bound $3s$ for $\chi^+(x)$, $x\in\gamma$. Let $\delta_0>0$ be such that $8\delta_0\leq \mbox{diam}(\gamma)$. There is $\delta_1>0$ such that for $\delta_0\geq \delta_1\geq \delta>0$ it holds that $\Lambda^+(x,\delta)>2s$.
Moreover, by Theorem \ref{t-multi} we have that
$$\lim_{n\to +\infty} \frac{1}{n} \log(H_\delta(x,n))=\Lambda^+(x,\delta)>2s $$
hence there are $n_0>0$ and $y\in \gamma$ such that $\dist(x,y)\leq\delta_1$ such that for all $n\geq n_0$
$$\frac{1}{n} \log(H_\delta(x,n))>s \Longrightarrow \dist_H(F^n(x),F^n(y))\geq e^{sn}\dist(x,y)\, .$$
Let $x_0,x_1\in\gamma$ be such that $\dist(x_0,x_1)\leq \delta_0$. Let $n_1>0$ be such that
$$\dist_H(F^{n_1}(x_0),F^{n_1}(x_1))\geq e^{n_1s}\dist(x_0,x_1)\geq 4\delta_0 \, $$
while $$\dist_H(F^{n}(x_0),F^{n}(x_1))<4\delta_0, \, \mbox{ for } 0\leq n<n_1\,.$$
Rename $x_0$ as $x_{00}$, $x_1$ as $x_{11}$ and find $x_{01}\in \gamma$ such that $\dist_H(F^n(x_{00}),F^n(x_{01}))<\delta_0$ for every $n:\,0\leq n\leq n_1$.
Similarly, find $x_{10}\in\gamma$ such that $\dist_H(F^n(x_{10}),F^n(x_{11}))<\delta_0$ for every $n:\,0\leq n\leq n_1$.
There is $n_2$ such that
$$\dist_H(F^{n_2}(x_{00}),F^{n_2}(x_{01}))\geq e^{n_2s}\dist(x_{00},x_{01})\geq 4\delta_0 \, $$ and
$$\dist_H(F^{n_{2}}(x_{10}),F^{n_2}(x_{11}))\geq e^{n_2s}\dist(x_{10},x_{11})\geq 4\delta_0 \, . $$
Therefore the points $x_{00},x_{01},x_{10},x_{11}$ are $(n_2,4\delta_0)$ separated.

Proceeding in this way, as we have done in Theorem \ref{htop>0},  we  find $2^k$ points $(n,4\delta_0)$ separated
$\{x_{00\ldots 00},x_{00\ldots 01},\ldots ,x_{11\ldots 11}\},$
with $n$ depending only on the distance among the points $x_{00\ldots 00},x_{00\ldots 01},\ldots ,x_{11\ldots 11}$.
From this,  we  conclude that $h_{top}(F)$ is positive whenever there is a nontrivial continuum $\gamma\subset M$ with the property that $\chi^+(x)>0$ for every $x\in\gamma$.
$\blacksquare$

\begin{tabbing}
Universidade Federal do Rio de Janeiro, \hspace{1cm}\= Facultad de Ingenieria, \kill
 M. J. Pacifico, \> J. L. Vieitez, \\
Instituto de Matematica, \> Instituto de Matematica,\\
Universidade Federal do Rio de Janeiro, \> Facultad de Ingenieria,\\
C. P. 68.530, CEP 21.945-970, \> Universidad de la Republica,\\
Rio de Janeiro, R. J. , Brazil. \> CC30, CP 11300,\\
                                \> Montevideo, Uruguay\\
{\it pacifico@im.ufrj.br} \> {\it jvieitez@fing.edu.uy}
\end{tabbing}

\end{document}